\pdfoutput=1
\RequirePackage[english]{babel}
\documentclass[a4paper]{amsart}

\usepackage{amssymb,amscd,amsthm,hyperref}
\usepackage{enumerate}

\title[]{Groups definable in two orthogonal sorts}

\author[A.~Berarducci]{Alessandro Berarducci}
\address{Universit\`a di Pisa, Dipartimento di Matematica, Largo Bruno Pontecorvo 5, 56127 Pisa, Italy}
\email{berardu@dm.unipi.it}
\thanks{Partially supported by PRIN 2009WY32E8\_003 ``O-minimalit\`a, teoria degli insiemi, metodi e modelli non standard e applicazioni''.}
\author[M.~Mamino]{Marcello Mamino}
\address{CMAF Universidade de Lisboa\\
Av. Prof. Gama Pinto 2\\
1649-003 Lisboa, Portugal}
\email{mamino@ptmat.fc.ul.pt}
\thanks{Partially supported by Funda\c{c}\~ao para a Ci\^encia e a Tecnologia grant SFRH/BPD/73859/2010. We also acknowledge the support of FIRB2010, ``New advances in the Model Theory of exponentiation''.}
\date{Date: 13 March 2013}

\DeclareMathOperator{\R}{\mathbb R}
\def\N{\mathbb{N}}

\DeclareMathOperator{\Z}{\mathbb Z}
\DeclareMathOperator{\C}{\mathbb C}

\DeclareMathOperator{\acl}{acl}
\DeclareMathOperator{\SU}{SU}
\DeclareMathOperator{\tp}{tp}
\DeclareMathOperator{\dcl}{dcl}
\DeclareMathOperator{\few}{Few}
\DeclareMathOperator{\many}{Most}

\newcommand{\gn}[1]{\ulcorner #1 \urcorner} 
\newcommand{\ov}{\overline}

\def\sZ{Z}
\def\sR{R}
\newcommand{\mC}{\mathfrak C} 

\theoremstyle{plain}
\newtheorem{theorem}{Theorem}
\newtheorem*{theorem*}{Theorem}
\newtheorem*{corollary*}{Corollary}
\newtheorem{lemma}[theorem]{Lemma}
\newtheorem{observation}[theorem]{Observation}
\newtheorem{proposition}[theorem]{Proposition}
\newtheorem{corollary}[theorem]{Corollary}

\newtheorem{claim}{Claim}
\newenvironment{pclaim}[1][{\noindent\textsc{Proof}.}]{\begin{trivlist}
\item[\hskip \labelsep {#1}]}{\hfill \qed$\,_\textrm{claim}$ \end{trivlist}}

\newtheorem{fact}[theorem]{Fact}

\theoremstyle{definition}
\newtheorem{remark}[theorem]{Remark}
\newtheorem{definition}[theorem]{Definition}

\newtheorem{example}[theorem]{Example}
\newtheorem{exercise}[theorem]{Exercise}

\newtheorem{assumption}[theorem]{Assumption}

\numberwithin{theorem}{section}

\newcommand{\bt}{\begin{theorem}}
\newcommand{\et}{\end{theorem}}
\newcommand{\bl}{\begin{lemma}}
\newcommand{\el}{\end{lemma}}
\newcommand{\bfa}{\begin{fact}}
\newcommand{\efa}{\end{fact}}
\newcommand{\bexa}{\begin{example}}
\newcommand{\eexa}{\end{example}}
\newcommand{\bexe}{\begin{exercise}}
\newcommand{\eexe}{\end{exercise}}
\newcommand{\bprop}{\begin{proposition}}
\newcommand{\eprop}{\end{proposition}}
\newcommand{\bp}{\begin{proof}}
\newcommand{\ep}{\end{proof}}
\newcommand{\bc}{\begin{corollary}}
\newcommand{\ec}{\end{corollary}}
\newcommand{\bd}{\begin{definition}}
\newcommand{\ed}{\end{definition}}
\newcommand{\br}{\begin{remark}}
\newcommand{\er}{\end{remark}}

\def\st{\;|\;}

\newenvironment{acknowledgements}{{\bf Acknowledgements.}}{}

\begin{document} 

\begin{abstract} 
This work can be thought as a contribution to the model theory of group extensions. We study the groups $G$ which are interpretable in the disjoint union of two structures (seen as a two-sorted structure). We show that if one of the two structures is superstable of finite Lascar rank and the Lascar rank is definable, then $G$ is an extension of a group internal to the (possibly) unstable sort by a definable subgroup internal to the stable sort. In the final part of the paper we show that if the unstable sort is an o-minimal expansion of the reals, then $G$ has a natural Lie structure and the extension is a topological cover. 
\end{abstract} 

\maketitle

\section{Introduction}
This paper can be thought as a contribution to the model-theory of covers of groups in the spirit of \cite{Zilber2006,Hrushovski2011,Berarducci2010c}.  We assume some familiarity with the basic notions of model theory, but we recall some relevant definitions in \S \ref{orthogonal}. A good recent reference is \cite{Tent2012}. 

Given two structures $\sZ$ and $\sR$, let $(\sZ,\sR)$ be the two-sorted structure with a sort for $\sZ$ and another sort for $\sR$ in a disjoint language (no connections between the two sorts). Note that $Z$ and $R$ are then {\em fully orthogonal} in the following sense: any definable subset of $Z^m\times R^n$ is a finite union of sets of the form $A\times B$ with $A$ a definable subset of $Z^m$ and $B$ ad definable subset of $R^n$.

Our aim is to study the groups $G$ which are interpretable in $(\sZ,\sR)$, or equivalently definable in $(Z,R)^{eq}$ (see \S \ref{orthogonal} for the definitions). Obvious examples are the direct products of groups $H\times K$ with $H$ definable in $Z$ and $K$ definable in $R$. More generally one can have a quotient of $H\times K$ by a finite subgroup. There are however more interesting examples like the following.  

\bexa (\cite{Hrushovski2011})\label{univ}
The universal cover $f: G \to H$ of a real Lie group $H$ definable in an o-minimal expansion  $\sR$ of the real field is interpretable in $((\Z,+), \sR)$.  
\eexa

A few comments are in order. By \cite{Edmundo2007} the universal cover $f: G \to H$ of $H$ can be realized as a locally definable group and admits a definable section $s: H \to G$ (see also \cite{Berarducci2011}). In \cite[\S 8]{Hrushovski2011} it is showed that the bijection $G \to \ker(f) \times H$ induced by the section gives an intepretation of $(G,f,\sR)$ in the two-sorted structure $(\ker (f), \sR)$. 
On the other hand since $\ker(f) \cong \pi_1(H)$ is abelian and finitely generated, $\ker (f)$ is interpretable in $(\Z,+)$ and therefore $f: G \to H$ is interpretable in the two-sorted structure $((\Z,+),\sR)$. In the same way one shows that any cover of $H$ is interpretable in $((\Z,+),\sR)$. (See also  \cite[Prop. 3.1]{Berarducci2010c}.) 

Example \ref{univ} shows that a group $G$ interpretable in $(Z,R)^{eq}$ does not need to arise from a direct product. The next natural question is whether $G$ is always an extension of a group definable in one sort by a group definable in the other sort. We will show that this is indeed the case under a suitable stability assumption on $Z$, but let us first show that in full generality the question has a negative answer. To this aim we take $Z = R = (\R, +, <)$. So we have two structures  $Z$ and $R$ which are at the same time ``equal'' and orthogonal. There is of course no contradiction: indeed strictly speaking in $(Z,R)$ we only have an isomorphic copy of $Z$ and an isomorphic copy of $R$ with the isomorphism not definable in $(Z,R)$. 

\bexa \label{torus} Let $Z = R = (\R, +, <)$. There is a group $G$ definable in $(Z,R)^{eq}$ with no infinite definable subgroup internal to one of the two sorts. So in particular $G$ cannot be a definable extension of a group internal to one sort by a group internal to the other sort. \eexa
 
\bp (Based on \cite[Example 5.2]{Peterzil1999}) Take $G = (Z\times R)/\Lambda$ with $\Z^2 \cong \Lambda < Z \times R$ and $\Lambda$ in sufficiently generic position. Note that $\Lambda$ is not definable. However we can define $G$ in $(Z,R)^{eq}$ taking a definable set $X \subseteq Z\times R$ such that $X + \Lambda = G$ and $X \cap \Lambda$ is finite (a big enough square $X = [0,a] \times [0,a]$ will do) and identifying $G$ with $X/\Gamma$ (where $X/\Lambda$ is the quotient of $X$ by the equivalent relation ``to be in the same coset''). Since we only need a finite portion of $\Lambda$ to define $X/\Lambda$ we obtain a definition in $(Z,R)^{eq}$. This is exactly the example in \cite{Peterzil1999} except that in that paper the authors work with only one sort (which amounts to have the identity map from $Z$ to $R$ at disposal). They prove that in the one-sort setting $G$ has no definable proper infinite subgroups. This holds {\em a fortiori} in the two-sorted setting since we have fewer definable sets.  Thus clearly $G$ has no infinite subgroups internal to one of the two sorts. 
\ep 

In Example \ref{torus} it is important to have the order relation $<$ in the language, so the structures are unstable (in the model-theoretic sense). If we work with the stable structures $Z = R = (\R,+)$ the argument breaks down since in this case we are not able to define $G = (Z\times R)/\Lambda$ in $(Z,R)^{eq}$. We will show that under a suitable stability assumption on the $Z$-sort any group interpretable in $(Z,R)$ is an extension of a group interpretable in $R$ by a group interpretable in $Z$. Let us recall that in a stable theory the $\SU$-rank coincides the $U$-rank or ``Lascar rank'' \cite{Tent2012,Lascar1976}. Our main result is: 

\begin{theorem*}[See \ref{main2}]Let $\sZ$ be a superstable structure of finite $\SU$-rank and assume that the $\SU$-rank is definable. Let $\sR$ be an arbitrary structure. Given a group $(G,\cdot)$ definable in $(\sZ,\sR)^{eq}$, there is a $\sZ$-internal definable normal subgroup $\Gamma\lhd G$ such that $G/\Gamma$ is $\sR$-internal.
\end{theorem*}

Note that in any superstable structure $Z$ of $\SU$-rank $1$ (for instance $(\C,+,\cdot)$ or $(\Z,+)$, or $(\R, +)$), the $\SU$-rank is definable (see \cite{Pillay1996}), and therefore $Z$ satisfies the assumption of the theorem. 

The subgroup $\Gamma$ will in general depend on how $G$ sits in the ambient space $(Z,R)^{eq}$ and not only on the definable isomorphism type of $G$. In particular $\Gamma$ is neither the minimal nor the maximal $Z$-internal normal subgroup such that $G/\Gamma$ is $R$-internal. For instance if $G$ is the universal cover of the circle group $\R/\Z$, then $G$ can be naturally interpreted in $((\Z,+), \R)$ by \cite{Hrushovski2011}, but it has no minimal or maximal $Z$-internal normal subgroup. In this example $\Gamma$ is $Z$-internal if and only if $2\Gamma$ is such, and $G/\Gamma$ is $R$-internal if and only if $G/2\Gamma$ is such, so there is no reason to prefer $\Gamma$ over $2\Gamma$.  

The subgroup $\Gamma$ is easier to describe if $G$ is definable in $(Z,R)$ rather than $(Z,R)^{eq}$. In this case we have $G \subseteq Z^m \times R^n$ for some $m,n\in \N$ and we can consider the projection $\pi_R : Z^m \times R^n \to R^n$. We then define: 
\[
\begin{array}{lll}
\Gamma & = & \left\{g\in G :  (\many y) (\many x) \big(\pi_R(xg^y) = \pi_R(g^yx) = \pi_R(x)\big)\right\}
\end{array}
\]
where $g^y = ygy^{-1}$ and $(\many y) \phi(y)$ means that the projection on $Z^m$ of the set of $y\in G$ such that $\phi(y)$ fails has lower $\SU$-rank than the projection of the whole of $G$. 

The definition of $\Gamma$ when $G$ is definable in $(Z,R)^{eq}$ is similar, but we need to redefine $\pi_R$ to give meaning to the formula. To do this we will first show that there is a finite-to-one function $f$ from $(Z,R)^{eq}$ to $Z^{eq} \times R^{eq}$ (uniform in each sort). Composing with the projection from $Z^{eq} \times R^{eq}$ to $R^{eq}$ we obtain the desired substitute for $\pi_R$ and the same definition of $\Gamma$ will then work. 
 
In \S \ref{o-minimal} we prove: 

\begin{theorem*}[See \ref{c0}] If $Z$ is an arbitrary structure and $R$ is o-minimal, then every group $G$ definable in $(Z,R)$ admits a unique ``$t$-topology'' in analogy with the o-minimal case (treated in \cite{Pillay1988}). 
\end{theorem*}

In particular, if $R$ is based on the reals, then $G$ has a natural Lie group structure. Combining Theorem \ref{c0} and Theorem \ref{main2} we then obtain: 

\begin{corollary*}[See \ref{cover}]
If $R$ is o-minimal and $Z$ is superstable of finite Lascar rank, any group definable in $(Z,R)$ is a cover of a group definable in $R$. 
\end{corollary*}
Here by ``cover'' we mean a definable morphism which is continuous and open in the t-topology and has a discrete kernel.
Note that there are extensions of groups definable in $\R = (\R,+,\cdot)$ by $(\Z,+)$ that are not covers (see \cite[Theorem 3.12]{Berarducci2010c}). Among all the extensions, only the covers will be definable in $((\Z,+),\R)^{eq}$. The following example may be instructive: 

\bexa (\cite[Theorem 3.12]{Berarducci2010c}) There is an exact sequences $0 \to (\Z,+) \to G \stackrel{f}\to H \to 0$ such that $G \cong (\R,+)$, $H \cong \R/\Z$, and $f(1/n) \in H$ does not converge for $n\to \infty$ (where $1$ is any fixed element of $G$). So in particular we cannot put a compatible topology on $G$ making $f$ into a covering. (Hence the morphism $f: G \to H$ is not interpretable in $((\Z,+), \R)$).  
\eexa

The paper is organized as follows. In \S \ref{orthogonal} we recall some model-theoretic definitions that are needed in the paper: fully orthogonal sets, internal sets, codes and imaginaries. Then we analyze the bearing of these notions for definable sets in $(Z,R)^{eq}$. 
In \S \ref{imaginaries} we define a finite-to-one map from $(Z,R)^{eq}$ to $Z^{eq} \times R^{eq}$ and establish its properties. In \S \ref{dim-types} we assume that $Z$ is superstable of finite $\SU$-rank and we define the dimension of a type in $(Z,R)^{eq}$ as the $\SU$-rank of its projection to $Z^{eq}$ (using the projections studied in \S \ref{imaginaries}). From the existence of non-forking extensions in $Z^{eq}$ we derive a similar results for $(Z,R)^{eq}$ (Proposition \ref{extension}). We will sometime consider types over ``unbounded'' sets of parameters, namely sets of parameters whose size depends on the model. In particular in Corollary \ref{strong-ext} we have a set of parameters including the whole of $R$ and we are nevertheless able to find a suitable realization of the type (this will play a crucial role in the proof of the main theorem). In \S \ref{dim-sets} we define the dimension of definable sets in $(Z,R)^{eq}$ as the maximal dimension of the types of its elements. We can equivalently define the dimension of a definable set in $(Z,R)^{eq}$ as the $\SU$-rank of its ``projection'' to $Z^{eq}$. However in order to prove the invariance of the dimension under definable bijections it is more convenient to use the approach via types (since projections do not commute with bijections). In most of the lemmas we do not really need the stability assumption but only the additivity of the $\SU$-rank (which holds also in supersimple theories of finite $\SU$-rank). However superstability is used in Proposition \ref{sameSU} and to prove the density property in Theorem \ref{dim}. In \S \ref{quantifier} we introduce the quantifier ``for most $x$'' based on the dimension on $(Z,R)^{eq}$. This gives us a convenient notation to define the subgroup $\Gamma$ whose existence is asserted in the main theorem. It is important to assume that in $Z$ the $\SU$-rank is definable in order to ensure that the quantifier ``for most $x$'' is a definable operation. Finally in \S \ref{main-section} we prove the main theorem, in \S \ref{conclusion} we derive some consequences, and in \S \ref{o-minimal} we consider the case when $R$ is o-minimal. 

\begin{acknowledgements} The results of this paper were obtained through a long chain of successive generalizations in the course of which the hypothesis were progressively weakened. Part of this process was stimulated by conversations with Anand Pillay on the occasion of the meetings ``Model Theory in Algebra, Analysis and Arithmetic'' (Cetraro, 2-6 July 2012) and ``Model Theory: Groups, Geometry, and Combinatorics'' (Oberwolfach 6-12 Jan. 2013). We also thank Ya'acov Peterzil for his suggestion to study the subgroup $\Gamma \cdot C_G(\Gamma)$ (see Corollary \ref{centr}). Preliminary versions of the results were presented at the ``Konstanz-Naples Model Theory Days'' (Konstanz 6-8 Dec. 2012) and at the Oberwolfach meeting. 
\end{acknowledgements}


\section{Orthogonality and internal sets} \label{orthogonal} 
Given a complete first order theory $T$ we usually denote by $\mC$ the monster model of $T$ (see \cite{Tent2012}). We can think of $\mC$ as a proper class model which is $\kappa$-saturated for every cardinal $\kappa$.  Every {\em small} (i.e. set-sized) models of $T$ can be embedded in $\mC$ as an elementary substructure, so we can assume that all the models that we are interested are elementary substructures of $\mC$. 

Two definable subsets $\sZ$ and $\sR$ of $\mC$ are {\em fully orthogonal} if every definable subset $X$ of $\sZ\times R$ is a finite union $\bigcup_{i\in I} U_i \times V_i$ of ``rectangles'' $U_i\times V_i$ with $U_i$ a definable subset of $X$ and $V_i$ a definable subset of $\sR$. It can be shown that if $Z$ and $R$ are fully orthogonal, then so are $Z^m$ and $R^n$, but to avoid the verification we can as well include this condition in the definition of full orthogonality. 

A definable set $X$ in $\mC$ is {\em stably embedded} if every subset of $X$ definable with parameters from $\mC$ is definable with parameters from $X$. In this case we consider $X$ as a structure on its own right with a symbol for each $\emptyset$-definable subset of $X^n$ for any $n$. 

Given two definable sets $X$ and $V$ in $\mC$, $X$ is said to be internal to $V$, or {\em ${V}$-internal}, if $X$ is in the definable closure of $V$ and a finite set of parameters.  Equivalently there is a definable surjection from ${V}^n$ to $X$ for some $n$ (see for instance \cite[Lemma 10.1.4]{Tent2012}). 
We recall that $\mC$ has {\em elimination of imaginaries} if, by definition, every definable set $X$ in $\mC$ has a code, where {\em a code for $X$} is a finite tuple $c$ of elements of $\mC$ such that for every automorphism $f$ of $\mC$ we have that $f$ fixes $c$ if and only if $f$ fixes $X$ setwise. All the codes for the same set $X$ are interdefinable and $X$ is definable over any of its codes. We follow the common convention of denoting by $\gn{X}$ a code for $X$. 

As usual $\mC^{eq}$ denotes the expansion of $\mC$ with imaginary sorts: for each $n\in \N$ and each $\emptyset$-definable equivalence relation $E$ on $\mC^n$ we have a sort $S_E$ interpreted as $\mC^n/E$ together with the natural projection $\pi_E: \mC^n\to \mC^n/E$. Given $a\in \mC^n$ the $E$-equivalence class of $a$ can be seen in two ways: as a definable subset $[a]_E$ of $\mC^n$, or as an element $a/E$ of $S_E$. The structure $\mC^{eq}$ has elimination of imaginaries. In particular the definable set $[a]_E$ can be coded by the element $a/E\in S_E$. 

The advantage of working in $\mC^{eq}$ is that quotients of definable sets by definable equivalence relations become definable. It follows in particular that a group is interpretable in $\mC$ if and only if it is isomorphic to a group definable in $\mC^{eq}$. 

\begin{assumption}\label{sym}
In the rest of the paper we assume that $\sZ$ and $\sR$ are stably embedded and fully orthogonal definable subsets of a monster model $\mathfrak C$. The subsets of $\sZ^m \times R^n$ definable in $\mC$ are exactly (up to a natural identification) the definable sets in the two sorted structure $(\sZ,\sR)$ (with no connections between the two sorts). Unless otherwise stated in the sequel  by ``definable'' we mean definable in $(Z,R)^{eq}$ (possibly with parameters from the monster model).  
\end{assumption}

By symmetry all the results depending only on Assumption \ref{sym} hold with the roles of $Z$ and $R$ interchanged. This applies in particular to the following: 

\bl \label{finite} Let $(X_t )_{t \in Z^m}$ be a definable family of subsets of $R^n$ indexed by $Z^m$. Then $\{X_t : t \in Z^m\}$ is finite. More generally the same holds for a definable family $(X_t)_{t \in Y}$ of subsets of an $R$-internal set $X$ indexed by a $Z$-internal set $Y$ (where $X$ and $Y$ are definable sets in  $(Z,R)^{eq}$). \el
\bp For the first part consider the definable set $X = \{(t,x) : x \in X_t\} \subseteq Z^m \times R^n$. By full orthogonality we can write it as a finite union of definable sets of the form $A \times B$ with $A \subseteq R^n, B \subseteq Z^m$ and the desired result follows at once. For the second part let $f: R^n \to X$ and $g: Z^n \to Y$ be definable surjective maps witnessing internality to the respective sorts. 
By the first part $\{f^{-1}(X_{g(z)}) : z \in Z^m\}$ is finite. So $\{X_t: t \in Y\}$ is also finite. 
\ep

\bprop \label{intr} 
For a definable set $X \subseteq Z^m \times R^n$ the following are equivalent: 
\begin{enumerate}
\item $X$ is $Z$-internal. 
\item The projection of $X$ on the $R$-coordinates is finite. 
\item There is a definable bijection from $X$ to a definable subset of $Z^{m+1}$.  
\end{enumerate}
\eprop 
\bp
(1) implies (2): Suppose $X$ is $Z$-internal. Then there is a definable surjective map $f: Z^k \to X$ for some $k$. 
Composing with the projection $\pi_R: Z^m \times R^n \to R^n$ we obtain a map $f$ from $Z^k$ to $R^k$ whose image is finite by Lemma \ref{finite}. On the other hand the image of $f$ coincides with  the projection of $X$ onto the $R$-coordinates.

(2) implies (3): Suppose $X \subseteq Z^m \times F$ where $F \subseteq R^n$ is finite. Fix a bijection $f$ from $F$ to a finite subset of $Z$. Then $f$ induces a bijection from $X$ to a definable subset of $Z^{m+1}$. 

(3) implies (1): Assume (3). Then clearly there is a definable surjective map from $Z^{m+1}$ to $X$, so $X$ is $Z$-internal. 
\ep 

\bl \label{RandZ} Let $X$ be a definable set in $(Z,R)^{eq}$ which is both $Z$-internal and $R$-internal. Then $X$ is finite. \el
\bp By the hypothesis there are $m,n\in \N$ and definable surjective functions $f: Z^m \to X$ and $g: R^n \to X$. For $x\in Z^m$ let $H(x) = g^{-1}(f(x)) \subseteq R^n$. By Lemma \ref{finite} the family of sets $\{H(x): x \in Z^m\}$ is finite. Since distinct elements of $X$ have disjoint preimages through $g$, it follows that $X$ is finite. 
\ep

Let us recall that, given a set $A$ of parameters in some structure, the {\em definable closure} $\dcl(A)$ is the set of points which are definable over $A$ and the {\em algebraic closure} $\acl(A)$ is the set of points which belong to some finite set definable over $A$. Clearly $\acl(A) \supseteq \dcl(A)$.  
  
\bl \label{box} Let $A$ be a set of parameters from $(Z,R)^{eq}$. We have:   
\begin{enumerate}
\item
Let $X \subseteq Z^m \times R^n$ be definable over $A$. Then we can write $X = \bigcup_{i=1}^k U_i \times V_i$ with $U_i \subseteq \sZ^m$ definable over $\acl(A) \cap Z^{eq}$ and $B_i \subseteq \sR^m$ definable over $\acl(A) \cap R^{eq}$.
In particular $X$ is definable over $(\acl(A)\cap Z^{eq}) \cup (\acl(A) \cap R^{eq})$. 

\item Let $a$ be an element of $(Z,R)^{eq}$. Then $a$ is definable over $(\acl(a) \cap Z^{eq}) \cup (\acl(a) \cap R^{eq})$.  

\item Let $X$ be a definable subset of some sort of $Z^{eq}$ and suppose that $X$ is definable over $A$. Then $X$ is definable over $\dcl(A) \cap Z^{eq}$. 

\item For every set of parameters $A$ from $(Z,R)^{eq}$ we have $\acl(A \cup R) \cap Z^{eq} = \acl(A) \cap Z^{eq}$.  

\item The type of an element of $Z^{eq}$ over $A\cup R$, is implied by its type over $\acl(A) \cap Z^{eq}$. 

\item Given a tuple $b$ from $(Z,R)^{eq}$, $\tp(b/R)$ is implied by $\tp(b/\dcl(b)\cap R^{eq})$. More generally $\tp(b/A \cup R)$ is implied by $\tp(b/A \cup (\dcl(b) \cap R^{eq}))$ . 
 \end{enumerate} 
\el 
Note that in (5) and (6) a type over the big set of parameters $R = R(\mC)$ is implied by a type over a small set of parameters. 
\bp 
(1) By considering the sections $X_r = \{x\in Z^n : (x,r) \in X\}$ with $r\in R^m$ we obtain a finite boolean algebra of definable subsets of $Z^n$. The atoms of this boolean algebra are permuted by any automorphism of the monster model fixing $X$ setwise, so they can be coded by elements in $\acl(A)\cap Z^{eq}$ (the codes can be taken in $Z^{eq}$ by stable embeddedness). It follows that each set in the boolean algebra is definable over $\acl(A) \cap Z^{eq}$. Similarly, considering the sections over $Z^n$, we obtain a finite boolean algebra of subsets of $R^m$. If $U \subseteq Z^m$ and $V \subseteq R^n$ are atoms of the respective boolean algebras, then $U$ is definable over $\acl(A)\cap Z^{eq}$, $V$ is definable over $\acl(A)\cap R^{eq}$ and $U\times V$ is either contained or disjoint from $X$. The desired result follows. 

(2) Let $X\subseteq Z^m\times R^n$ be the equivalence class corresponding to $a \in (Z,R)^{eq}$ and apply (1). 

(3) By stable embeddedness $X$ is definable with parameters from $Z^{eq}$ so it has a code $\gn{X}$ in $Z^{eq}$. 
On the other hand since $X$ is definable over $A$, we have $\gn{X} \in \dcl(A)$. So $X$ is definable over $\dcl(A) \cap Z^{eq}$. 

(4) Without loss of generality we can assume $A=\emptyset$. Let $b\in \acl(R) \cap Z^{eq}$. Then there is a tuple $r$ from $R$ and an algebraic formula $\phi(x,r)$ such that $\phi(b,r)$ holds. Let $N\in \N$ be the cardinality of $X_r = \{x : \phi(x,r)\}$. We can assume that for all tuples $r'$ from $R$ of the same length as $r$, the set $X_{r'} = \{x: \phi(x,r')\}$ has cardinality at most $N$. Each $X_{r'}$ is a subset of a given sort of $Z^{eq}$ (the sort of $b$), and the family of these sets is indexed by $R^k$ for some $k$ (the length of the tuple $r$). By Lemma \ref{finite} it follows that the family of sets $\{X_{r'}\}_{r' \in R^k}$ is finite, and since each of them is finite, the union $\bigcup_{r'} X_{r'}$ is finite. Now it suffices to observe that this union is $\emptyset$-definable and contains $b$. 

(5) We can assume $A= \emptyset$. Let $e\in Z^{eq}$ and let $\phi(x,r) \in \tp(b/R)$. By point (3) (with $A = \{r\}$) $\phi(x,r)$  is equivalent to a formula $\psi(x)$ with parameters from $\dcl(r) \cap Z^{eq} \subseteq \acl(\emptyset) \cap Z^{eq}$, where the inclusion follows from point (4). 

(6) It suffices to prove the case $A=\emptyset$. Let $\phi(x,r)\in \tp(b/R)$ where $r$ is a tuple from $R$, say $r\in R^k$. 
So $r \in Y:= \{y\in R^k : \phi(b,y)\}$. By stable embeddedness $Y$ can be defined with parameters from $R$, so it has a code $\gn{Y}$ in $R^{eq}$. On the other hand we must also have $\gn{Y}\in \dcl(b)$, so $\phi(b,y)$ is equivalent to a formula $\psi(y)$ definable over $\dcl(b)\cap R^{eq}$. To conclude it suffices to observe that the formula $\forall y(\psi(y)\to \phi(x,y))$ belongs to $\tp(b/\dcl(b)\cap R^{eq})$ and implies $\phi(x,r)$ (take $y=r$).  
\ep

\section{Imaginaries} \label{imaginaries}

In this section we define a finite-to-one function from $(Z,R)^{eq}$ to $Z^{eq} \times R^{eq}$ and study its properties. 

\bl \label{unicode} Given a definable family $(X_t)_{t \in Y}$ of definable susets of $Z^m$ indexed by a definable set $Y$ in $(Z,R)^{eq}$, there is a uniform family $(\gn{X_t})_{t\in Y}$ of codes in the following sense:  
\begin{enumerate}
\item For each $t\in Y$ the set $X_t\subseteq Z^m$ is coded by $\gn{X_t} \in Z^{eq}$ and the function $t\mapsto \gn{X_t}$ from $Y$ to the appropriate sort of $Z^{eq}$ is definable. 
\item For all $t,t' \in Y$ we have $\gn{X_t} = \gn{X_{t'}}$ if and only if $X_t = X_{t'}$. 
\end{enumerate}
\el 
\bp Given $t\in Y$, by stable embeddedness there is a formula $\varphi(-,b)$ with parameters $b$ from $Z$ which defines $X_t$ (where ``$-$'' is the free variable of the formula). By compactness there is a finite collection of formulas $\phi_1, \ldots, \phi_k$ such that for every $t\in Y$ there is $i\leq k$ and a tuple $b$ from $Z$ such that $X_t$ is defined by $\phi_i(-,b)$. Let $b E_i b'$ if and only if $\phi_i(-,b)$ is equivalent to $\phi_i(-,b')$ and let $a = b/E_i \in Z^{eq}$ be the corresponding imaginary element. Define $f(t) = (i,b/E_i)$, where $X_t$ is defined by $\phi_i(-,b)$ and $i$ is minimal such that there exists such a tuple $b$. Identifying the indexes $1,\ldots, k$ with tuples from $Z$ we can consider $f(t)$ as an element of the appropriate sort of $Z^{eq}$ and define $\gn{X_t} = f(t)$. 
\ep 

\begin{definition} \label{aZ} Let $S_E = (Z^m \times R^n)/E$ be a sort of $(Z,R)^{eq}$. For $a\in S_E$ let $\pi_E^{-1}(a) \subseteq Z^m \times R^n$ be the equivalence class corresponding $a$ and let $\pi_Z(\pi_E^{-1}(a)) \subseteq Z^m$ be its projection on $Z^m$. For $a\in S_E$ define $$a^Z = \gn{\pi_Z(\pi_E^{-1}(a))} \in Z^{eq}$$ 
where the codes are chosen uniformly as in Lemma \ref{unicode} (so $a\mapsto a^Z$ is definable). 
Note that when $a\in Z^m \times R^n$, then $a \mapsto a^Z\in Z^m$ is the natural projection. Similarly we define $a^R = \gn{\pi_R(\pi_E^{-1}(a))}\in R^{eq}$.
\end{definition}

\bd \label{XZ} Given a definable set $X$ in $(Z,R)^{eq}$ define $$X^Z=\{a^Z : a\in X\}$$ where $a\mapsto a^Z$ is given by Definition \ref{aZ}. 
Then $X^Z$ is a definable subset of some sort of $Z^{eq}$ and by stable embeddedness it is definable in $Z^{eq}$. Similarly we define $X^R = \{a^R : a\in X\}$. 
\ed 

\begin{lemma} \label{inteq} Let $X$ be a definable set in $(Z,R)^{eq}$. Then $X$ is $R$-internal if and only if $X^Z$ is finite. 
Similary $X$ is $Z$-internal if and only if $X^R$ is finite. 
\end{lemma} 
\bp Let $X \subseteq S_E = (Z^m \times R^n)/E$ and suppose that $X^Z$ is finite. Unraveling the definitions this means that $\{\pi_Z(\pi_E(x))^{-1} : x\in X\}$ is a finite family of non-empty subsets of $Z^m$. So there is a finite subset $b$ of $Z^m$ which meets all these sets. Thus $X \subseteq \pi_E(b \times R^n) \subseteq \dcl(bR)$ and $X$ is $R$-internal. 

Conversely suppose that $X$ is $R$-internal. Then $\{\pi_Z(\pi_E(x))^{-1} : x\in X\}$ is a family of subsets of $Z^m$ indexed by an $R$-internal set, so it must be finite by Lemma \ref{finite}. 
\ep 

\begin{lemma} \label{LaZaR} 
Let $a\in (Z,R)^{eq}$. Then $a$ is algebraic over~$(a^Z, a^R) \in Z^{eq}\times R^{eq}$. 
\end{lemma}
\bp By Lemma \ref{inteq} the set 
$\{x\;|\;x^Z=a^Z\}\cap\{x\;|\;x^R=a^R\}$ is internal to both $Z$
and~$R$, and therefore it is finite by Lemma \ref{RandZ} . Since $a$ belongs to this finite set, 
the Lemma is established.
\ep 

\begin{definition} \label{DefAZ} Given $a\in (Z,R)^{eq}$ define $\ov{a}^Z = \acl(a) \cap Z^{eq}$. More generally, 
given a set $A$ of parameters from $(Z,R)^{eq}$, define ${\ov A}^Z=\acl(A) \cap Z^{eq}$. 
Similarly define $\ov{a}^R$ and ${\ov A}^R$.
\end{definition}

\begin{observation} \label{OaZaR} 
Let $a\in (Z,R)^{eq}$. Then: 
\begin{enumerate}
\item $\acl(a^Z) \cap Z^{eq} = \acl(a) \cap Z^{eq} = \ov{a}^Z$. 
\item $a$ is definable over ${\ov a}^Z\cup{\ov a}^R$.
\end{enumerate}
\end{observation}
\begin{proof} 
By Lemma \ref{LaZaR} we have $\acl(a) = \acl(a^Za^R)$, and by Lemma \ref{box}(4) the element $a^R\in R^{eq}$ does not contribute to the algebraic closure in $Z^{eq}$, so the first part is established. The second part is Lemma \ref{box}(2) in the new notation. \end{proof}

\section{dimension of types}\label{dim-types}

In the sequel we assume that the theory of $Z$ is superstable of finite $\SU$-rank and that the $\SU$-rank is definable. In $Z$ we have a good notion of dimension given by the $\SU$-rank, and we want to define a dimension on $(Z,R)^{eq}$. We recall the following:

\begin{fact} Let $M$ be a superstable structure of finite $\SU$-rank. Then: 
\begin{enumerate}
\item (Additivity) For every $a,b\in M^{eq}$ and every set $A$ of parameters from $M^{eq}$ we have $\SU(ab/A) = \SU(a/A) + \SU(b/aA)$. 
\item (Definability) If $M$ has $\SU$-rank $1$, then the $\SU$-rank is definable, namely for every definable family $(X_t)_{t \in Y}$ of sets in $M^{eq}$ and every $k\in \N$, the set $\{t\in Y: \dim(X_t) = k\}$ is definable. 
\end{enumerate}
\end{fact}
Part (1) is well known and follows from Lascar's inequalities \cite[Theorem 8]{Lascar1976}. For part (2) see \cite[Corollary 5.11]{Pillay1996}. 

\begin{definition} Given $a\in (Z,R)^{eq}$ define
$$\dim(a/A) = \SU(a^Z/{\ov A}^Z),$$ where $\ov{A}^Z = \acl(A) \cap Z^{eq}$ (Definition \ref{DefAZ}) and $a^Z$ is as in Definition \ref{aZ}.  
\end{definition}

\begin{lemma}\label{th-additivity-points} Work in $(Z,R)^{eq}$. We have: 
\begin{enumerate} 
\item If $b\in \acl(A)$, then $\dim(b/A) = 0$. 
\item $\dim(ab/A) = \dim(a/A) + \dim(b/aA)$. 
\item If $r \subset R^{eq}$, then $\dim(a/Ar) = \dim(a/A)$. 
\end{enumerate}
\end{lemma}
\begin{proof}
(1)  Assume $b\in \acl(A)$. Then $b^Z\in \ov{A}^Z$ and $\dim(b/A) = \SU(b^Z/\ov{A}^Z) = 0$. 

Part (2) follows from the additivity formula for the $\SU$-rank in~$Z^{eq}$,
observing that $\ov{aA}^Z = \acl(a^Z{\ov A}^Z)$, where the algebraic
closure is inside~$Z^{eq}$.

For (3) it suffices to observe that $\ov{Ar}^Z = \ov{A}^Z$. 
\end{proof}

\begin{proposition} \label{extension}
Let $a$ be an element of $(Z,R)^{eq}$ and $A\subseteq B$ be sets of parameters in $(Z,R)^{eq}$. Then
there is~$b\in (Z,R)^{eq}$ such that $\dim(b/B)=\dim(a/A)$ and $\tp(b/A)=\tp(a/A)$. 
\end{proposition}
\begin{proof}
By Lemma \ref{LaZaR} there is an algebraic formula $\phi(x,a^Z,a^R)$ (possibly with additional parameters from $A$) which isolates the type of $a$ over $a^Za^RA$. In particular for each $\psi(x) \in \tp(a/A)$ we have:  
$$\models \forall x(\phi(x,a^Z,a^R) \to \psi(x)). \eqno(\dagger)$$ 
Let $e\in Z^{eq}$ be a realization of a non-forking extension of $\tp(a^Z/{\ov A}^Z)$ to~${\ov
B}^Z$. This means that $\tp(e/\ov{A}^Z) = \tp(a^Z/\ov{A}^Z)$ and $\SU(e/{\ov B}^Z) = \SU(e/{\ov A}^Z)$.  By Lemma \ref{box}(4) we have $\tp(e/A \cup R) = \tp(a^Z/A \cup R)$. 
It follows that we can replace $a^Z$ with $e$ in $(\dagger)$ and obtain
$$\models \forall x(\phi(x,e,a^R) \to \psi(x)).$$
We have thus proved that $\phi(x,e,a^R)$ isolates $\tp(a/A)$. 
Moreover $\phi(x,e,a^R)$ is an algebraic consistent formula, since this fact is a property in the type of $e$ inherited from $a^Z$. Choose $b\in (Z,R)^{eq}$ such that $\phi(b,e,a^R)$ holds. We can assume that $\phi(x,u,v)$ implies $u=x^Z$ and $v = x^R$, since otherwise we can add these conditions to the formula (using the fact that $x\mapsto x^Z$ and $x\mapsto x^R$ are definable functions). So $b^Z = e$ and $b^R = a^R$.  
To conclude it suffices to observe that $\dim(b/B) = \SU(b^Z/\ov{B}^Z) = \SU(e/{\ov B}^Z) = \SU(e/{\ov A}^Z) = \SU(a^Z/{\ov A}^Z)$, where the last equality follows from the fact that $\tp(e/\ov{A}^Z) = \tp(a^Z/\ov{A}^Z)$. 
\end{proof}

Note that in the above Lemma the sets of parameters $A$ and $B$ must be small (with respect to the monster model $\mC$), so we cannot take $B = Z$ say. We can nevertheless obtain the following corollary, where $Z_0 \prec Z$ is a small model, but $R= R(\mC)$ is interpreted in the monster model.   

\bc \label{strong-ext} Let $a$ be an element of $(Z,R)^{eq}$. Then
there is~$b\in (Z,R)^{eq}$ such that $\dim(b/a)=\dim(a)$ and $\tp(b/\acl(Z_0\cup R))=\tp(a/\acl(Z_0 \cup R))$.
\ec

The idea is to use the fact that the type of $a$ over the big set of parameters $R$ is implied by the type of $a$ over the small set $\ov{a}^R \subset R^{eq}$. The details are as follows. 

\bp By Observation \ref{OaZaR} we have $\ov{a}^R = \acl(a) \cap R^{eq} = \acl(a^R) \cap R^{eq}$. 
By Lemma \ref{extension} (with $A = \ov{a}^R$ and $B = \ov{a}^R \cup a$) there is some $b$ with $\tp(b/\ov{a}^R) = \tp(a/\ov{a}^R)$ and $\dim(b/ \ov{a}^R \cup a) = \dim(a / \ov{a}^R)$. By Lemma \ref{th-additivity-points} the parameters from $R^{eq}$ do not contribute to the dimension, namely $\dim(a /\ov{a}^R) = \dim(a)$ and $\dim(b/ \ov{a}^R \cup a) = \dim(b/a)$. It remains to show that  $\tp(b/\acl(Z_0\cup R))=\tp(a/\acl(Z_0 \cup R))$. We can assume that the parameters from $Z_0$ are named by constants in the language, so we only need to prove worry about $R$. So assume $\phi(x,r)\in \tp(a/R)$, where $r \in R^{eq}$, and let us prove that $\phi(x,r) \in \tp(b/R)$. By Lemma \ref{box}(6) there is a formula $\psi(x)\in \tp(a/\ov{a}^R)$ which implies $\phi(x,r)$. Since $\tp(b/\ov{a}^R) = \tp(a/\ov{a}^R)$ we deduce that $\phi(b,r)$ holds. 
\ep 

\section{Dimension of definable sets}\label{dim-sets}

\begin{definition}\label{def-dim-set} Given a definable set $X$ in $(Z,R)^{eq}$ let
$$\dim(X) = \max_{a\in X} \dim(a/A)$$ where $A$ is any set of parameters
over which~$X$ is defined and $a$ ranges in the monster model. By Proposition \ref{extension} this does not depend on the choice of $A$. 
\end{definition}

\bprop \label{sameSU}
 Let $M$ be a model of a superstable theory of finite $\SU$-rank. Let $X$ be a definable subset of $M^n$ and let $Y$
be an $M$-definable subset of~$X$ of the same $\SU$-rank. Suppose that $X$ is defined using
parameters in a model~$M_0\prec M$. Then $Y$ has a point with coordinates in $\acl(M_0)$. The same holds for definable sets in $M^{eq}$ provided the $\SU$-rank is definable. 
 \eprop 
 \bp Let $m = SU(X) = SU(Y)$.  
By definition $\SU(Y) = \sup \{ SU(b/M) : b \in Y\} = m$.  By our assumption $m$ is a finite ordinal, so the sup is achieved. Choose $b \in Y$ with $\SU(b/M) = m$. 
Then $\SU(b/M_0) \geq SU(b / M) = m$. On the other hand  
$m = SU(X) \geq SU(b /M_0)$. Thus $tp(b/M)$ has the same $\SU$-rank of $tp(b / M_0)$. 
If $X$ is included in one of the real sorts $M^n$, this implies that
$tp(b /M)$ is finitely satisfiable in $M_0$. (Indeed for types over models in a superstable theory, the unique extension with the same $\SU$-rank coincides with the unique non-forking extension, which in turn coincides with the unique unique extension finitely satisfiable in the small model). 
Thus in this case $Y$ has a point with coordinates in $M_0$.

The case when $X$ is included in one of the imaginary sorts $M^n/E$ is easily deduced from the real case assuming the definability of the $\SU$-rank. To see this, let $D_i\subseteq M^n$ be the union of all the $E$-equivalence classes of $\SU$-rank $i$. By definability of the $\SU$-rank, the set $D_i$ is definable and $\SU(X) = \max_i \SU (X \cap \pi_E(D_i))$. So we can reduce to the case when all the $E$-equivalence classes have the same $\SU$-rank. Next observe that if all the equivalence classes have $\SU$-rank $i$, then by additivity of the $\SU$-rank we have $\dim(X) = \dim(\pi_E^{-1}(X)) - i$ and similarly for $Y$. By the real case we deduce that $\pi_E^{-1}(Y)$ intersects $\acl(M_0)^{eq}$, hence so does $Y$. 
 \ep

\br \label{SUeq}
Reasoning as in the last part of the proof, the definability of the $\SU$-rank for definable families in the imaginary sorts of $M^{eq}$ follows (using additivity) from the definability of the $\SU$-rank for definable families in the home sorts $M^n$. 
\er 

\bt \label{dim} Assume that the theory of $Z$ is superstable of finite $\SU$-rank that that the $\SU$-rank is definable. Put in $(Z,R)^{eq}$ the dimension function of Definition \ref{def-dim-set}. We have: 
\begin{enumerate}
\item (Additivity) \label{additivity} Given a definable surjective function $f: X \to Y$ with fibers of constant dimension $k$, we have $\dim(X) = \dim(Y) + k$.  
\item (Monotonicity) \label{monotonicity} $\dim(X\cup Y) = \max(\dim(X), \dim(Y))$. 
 \item \label{base} (Base) $\dim(X) = SU(X^Z)$. 
\item (Dimension zero) \label{dimzero} $\dim(X) = 0$ if and only if $X$ is $\sR$-internal.
\item (Density) \label{density} If $X$ is $\emptyset$-definable and $Y \subseteq X$ is a definable subset of the same dimension, then for every $Z_0 \prec Z$ the intersection of $\acl(Z_0 \cup R)^{eq}$ with $Y$ is non-empty. 
\item (Definability) \label{definability} Given $d\in \N$ and a definable
family~$(X_t : t\in Y)$ defined over~$A$, the set $Y_d :=
\{t\in Y\;|\;\dim(X_t)=d\}$ is definable over~$A$. 
\end{enumerate}
\et

\bp 
(\ref{additivity}) Assume for simplicity that $f: X \to Y$ is $\emptyset$-definable in $(Z,R)^{eq}$ and consider an element $b\in Y$ with $\dim(b) = \dim(Y)$ and an element $a\in f^{-1}(b)$ with $\dim(a/b) = \dim(f^{-1}(b)) = k$. Then 
\[
\begin{array}{llll}
\dim(X) \geq \dim(a) & = & \dim(a b) & \textrm{(by Lemma \ref{th-additivity-points})} \\
		& = 	& \dim(b) + \dim(a/b) & \\
		& = &\dim(Y) + k. & 
\end{array}
\]
Similarly, starting an element $a\in X$ with $\dim(a) = X$ and letting $b= f(a)$ we obtain the opposite inequality. 

(\ref{monotonicity}) is obvious. 

(\ref{base}) This follows from the additivity property applied to the definable function $x \mapsto x^Z$ after observing that for each $a\in X^Z$ the fiber $\{x \in X : x^Z = a\}$ is $R$-internal by Lemma \ref{inteq}, and therefore it has dimension zero by the easy part of (\ref{dimzero}). 
 
(\ref{dimzero})  Assume $\dim(X) = 0$. Then $\SU(X^Z) = 0$. So $X^Z$ is finite and by Lemma \ref{inteq} $X$ is $R$-internal. The converse is clear. 

(\ref{density}) The corresponding property for the $\SU$-rank in $Z^{eq}$ is given by Lemma \ref{sameSU}. Thanks to (\ref{base}) we can reduce to the $Z^{eq}$-case replacing $X$ with $X^Z$ and $Y$ with $Y^Z$. Indeed we only need to observe that if $Y^Z$ meets $\acl(Z_0 \cup R)$ then so does $Y$ (because $y\in \acl(y^Z,y^R)$ by Lemma \ref{LaZaR}). 

(\ref{definability}) Given a definable family $(X_t : t \in Y)$ in $(Z,R)^{eq}$ we have $\dim(X_t) = \SU({X_t}^Z)$, 
so we need to prove that $\SU({X_t}^Z) = d$ is a definable condition on $t$. Since ${X_t}^Z$ is definable in $Z^{eq}$ this is almost given by our assumptions on $Z$. The only problem is that the parameter $t$ does not range in $Z$ but in a sort of $(Z,R)^{eq}$. However, by stable embeddedness, the set $(X_t)^Z$ can be defined by a formula $\phi(x,b)$ with parameters $b$ from $Z$. Moreover by a compactness argument there is a single formula $\phi(x,y)$ and a definable function $t\mapsto b_t \in Z^m$ (for some $m\in \N$) such that for each $t\in Y$ the set ${X_t}^Z$ is defined by $\phi(x,b_t)$. This reduces the definability of $\dim$ in $(Z,R)^{eq}$ to the definability of $\SU$ in $Z^{eq}$. 
\ep 

\section{The quantifier $(\many x\in X)$}\label{quantifier}

\bd 
Given a definable set $X$ in $(\sZ,\sR)^{eq}$ and a formula $\phi(x)$ we define: 
\begin{itemize}
\item $(\few x \in X) \phi(x) \iff \dim(\{x\in X : \phi(x)\}) < \dim(X)$
\item $(\many x \in X) \phi(x) \iff (\few x \in X) \lnot \phi(x)$
\end{itemize}
\ed 

From Theorem \ref{dim} we immediately deduce the following:

\bl \label{inv} Fix $Z_0 \prec Z$. We have:
\begin{enumerate}\itemsep1em
\item $\models (\many x)\phi(x) \land (\many x) \psi(x) \iff (\many x)(\phi(x) \land \psi(x))$. 
\item If $f: X \to X$ is a definable bijection, then \\
$\models (\many x \in X) \phi(x) \iff (\many x \in X) \phi(f(x))$. 
\item Given a formula $\phi(x,y)$ there is a formula $\psi(y)$ such that\\ $\models (\many x \in X)\phi(x,y) \iff \psi(y)$. 
\item Let $\phi(x)$ be a formula, possibly with parameters, and let $X$ be $\emptyset$-definable in $(Z,R)^{eq}$. 
Suppose that for all points $a\in X$ algebraic over $Z_0 \cup R$ we have $\phi(a)$. 
Then $\models (\many x \in X) \phi(x)$.
\end{enumerate}
\el 
\bp We apply Theorem \ref{dim}. Part (1) follows from the additivity of dimension. Part (2) from the monotonicity. Part (3) from the definability. Point (4) from the density property. \ep 

\section{Main theorem}\label{main-section}

We are now ready to prove our main result. 

\bt \label{main2} Let $\sZ$ be a superstable structure of finite $\SU$-rank and assume that the $\SU$-rank is definable. Let $\sR$ be an arbitrary structure. Given a group $(G,\cdot)$ definable in $(\sZ,\sR)^{eq}$, there is a $\sZ$-internal definable normal subgroup $\Gamma\lhd G$ such that $G/\Gamma$ is $\sR$-internal.
\et 
\bp 
Let $S_E = (Z^m \times R^n)/E$ be the sort of $G$. 
Given $x\in G$, let $R(x) = \pi_R(\pi_E^{-1}(x)) \subseteq R^m$ where $\pi_R: Z^m \times R^n \to R^n$ is the projection. Note that $x^R = \gn{R(x)}$ according to Definition \ref{aZ}. 
Define
\[
\begin{array}{lll}
\Gamma & = & \left\{g\in G :  (\many y) (\many x) \big(R(xg^y) = R(g^yx) = R(x)\big)\right\}
\end{array}
\]
where $g^y = ygy^{-1}$ and $(\many -)$ stands for $(\many -\in G)$.

To prove that $\Gamma$ is a subgroup we use the invariance of the quantifier 
$(\many -)$ under bijections and the fact that the quantifier distributes over conjunctions (Lemma \ref{inv}). The details are as follows.  
Let $a,b\in \Gamma$. We need to show that $ab \in \Gamma$. We have:
\[ \begin{array}{lll}
 & b \in \Gamma  \\
\Longrightarrow & (\many y)(\many x) \big(R(xb^y) = R(x)\big) & \\
\Longrightarrow &  (\many y) (\many x)\big(R(xa^yb^y) = R(xa^y)\big) &\mbox{since $x \mapsto xa^y$ is a bijection}   \\
\Longrightarrow & (\many y) (\many x)\big(R(x(ab)^y) = R(xa^y)\big)   &    \end{array}\] 

On the other hand 
\[ \begin{array}{ll}
 & a \in \Gamma  \\
\Longrightarrow & (\many y)(\many x) \big(R(xa^y) = R(x)\big). 
   \end{array}\] 
So combining the two derivations we obtain

\[ \begin{array}{ll}
 & a,b \in \Gamma  \\
\Longrightarrow & (\many y)(\many x) \big(R(x(ab)^y) = R(x)\big)
   \end{array}\] 
and similarly one obtains 
\[ \begin{array}{ll}
 & a,b \in \Gamma  \\
\Longrightarrow & (\many y)(\many x) \big(R(x(ab)^y) = R(x) = R((ab)^y x) \big) 
   \end{array}\] 
witnessing $ab \in \Gamma$. 

Let us now prove that $a\in \Gamma$ implies $a^{-1}\in \Gamma$. (Here is where we need the fact that in the definition of $\Gamma$ we have both $xg^y$ and $g^yx$.) We have:
\[ \begin{array}{ll}
 & a \in \Gamma  \\
\Longrightarrow & (\many y)(\many x) \big(R(x a^y) = R(x)\big)  \\
\Longrightarrow & (\many y) (\many x) \big(R(x) = R(x(a^{-1})^y)\big) 
\end{array}
\]
where in the last implication we used the fact that $x \mapsto x(a^{-1})^y$ is a bijection. 
Similarly we obtain $$\models (\many y) (\many x) \big(R(x) = R((a^{-1})^yx)\big).$$
Together with the previous condition this yields $a^{-1}\in \Gamma$. 

We have thus proved that $\Gamma$ is a subgroup. Let us now check that $\Gamma$ is normal in $G$. To this aim note that $b\in \Gamma$ can be expressed in the form $(\many y)Q(b^y)$ where $Q$ is a suitable formula. If $z\in G$ we want to show that $b^z\in \Gamma$, namely $(\many y)Q((b^z)^y)$ holds. This follows from the fact that $(b^z)^y = b^{(zy)}$ and $y\mapsto zy$ is a definable bijection on $G$. 

We need to prove that $\Gamma$ is $Z$-internal. Without loss of generality  we can work in a $\aleph_1$-saturated model. Suppose for a contradiction that $\Gamma$ is not $Z$-internal. Then there is a countable infinite subset $\{g_i : i \in \omega\}$ of $\Gamma$ such that $R(g_i) \neq R(g_j)$ for $i\neq j$ (Lemma \ref{inteq}). By $\aleph_1$-saturation and the definition of $\Gamma$, there are $x,y\in G$ such that for every $i\in \omega$ we have $R(xg_i^y) = R(x)$. Fix such an $x,y$ and let $f: G \to G$ be the definable bijection $g\mapsto xg^y$. We then have $R(f(g_i)) = R(x)$ for all $i \in \omega$, namely each $g_i$ belongs to the definable set $S := \{g\in G : R(f(g)) = R(x)\}$. This set is in definable bijection with $S':= \{g\in G : R(g) = R(x)\}$ so it is $Z$-internal by Lemma \ref{inteq}. On the other hand, still by Lemma \ref{inteq}, $S$ cannot be $Z$-internal since it contains the infinite sequence $\{g_i: i\in \omega\}$ and $R(g_i) \neq R(g_j)$ for all $i\neq j$. This contradiction shows that $\Gamma$ is $Z$-internal. 

It remain to show that $G/\Gamma$ is $R$-internal, or equivalently (by Theorem \ref{dim}) that $\dim(G/\Gamma) = \dim(G) - \dim(\Gamma) = 0$. We can assume that $G$ is $\emptyset$-definable. Let $a\in G$ be such that $\dim(a) = \dim(G)$. By Corollary \ref{strong-ext} there is $b\in G$ be such that $tp(b/\acl(Z_0 \cup R))=tp(a/\acl(Z_0 \cup R))$ and $\dim(b/a) = \dim(a)$. It follows that $\dim(ab^{-1}) = \dim(G)$ (because $\dim(ab^{-1}) \geq \dim(ab^{-1}/a) = \dim(b^{-1}/a) = \dim(b/a) = \dim(a) = \dim(G)$). We claim that $ab^{-1} \in \Gamma$. Granted this we have $\dim(\Gamma) \geq \dim(ab^{-1})$ and we obtain $\dim(G) = \dim(\Gamma)$ as desired. To prove the claim we reason as follows. Since $a,b\in G$ have the same type over $\acl(Z_0 \cup R)$, for all $x,y\in G \cap \acl(Z_0 \cup R)$ 
we must have
\[R(a^yx) = R(b^yx).\]
Indeed if $R(a^yx) \neq R(b^yx)$, then taking $r\in R$ in the symmetric difference of $R(a^yx)$ and $R(b^yx)$, we obtain a formula with parameters in $r,x,y$ which distinguishes the types of $a$ and $b$. 

By Lemma \ref{inv} this implies 
\[ \models (\many y)(\many x) \big(R(a^yx) = R(b^yx) \big)\]
Hence, by the definable bijection $x \mapsto (b^{-1})^y x$ we also get
\[
 \models (\many x)(\many x) \big(R((ab^{-1})^{y}x) = R(x)\big).
\]
By the same method we obtain
\[
 \models (\many y)(\many x) \big(R(x(ab^{-1})^{y}) = R(x)\big).
\]
namely $ab^{-1}\in\Gamma$.
\ep

\section{Corollaries}\label{conclusion}

In this section we study the subgroup $\Gamma \cdot C_G(\Gamma)$ of $G$ and we show that it coincides with $G$ if $G$ is connected (i.e. it has no definable subgroups of finite index).  We need the following. 

\bl \label{Zsat} Let $f\colon X \to Y$ be a definable surjective function in $(Z,R)^{eq}$ and suppose that $Y$ is $R$-internal. Then there is an $R$-internal definable subset $U$ of $X$ such that $f|U\colon U \to Y$ is surjective. \el 
\begin{proof} We can assume that $X \subseteq Z^m\times R^n$ (if $X$ is a subset of some imaginary sort $(Z^m\times R^n)/E$ we reduce to this case by considering $f \circ \pi_E : Z^m \times R^n \to Y$). For $t\in Z^m$, let $X_t$
denote~$X\cap\{t\}\times R^n$. Then $X_t$ is $R$-internal and by Remark~\ref{finite} the set
$\left\{f(X_t):t\in Z^m\right\}$ is finite. It follows that there is a finite subset
$\Gamma$ of~$Z^m$ such that $\bigcup_{t\in\Gamma} f(X_t) =
\bigcup_{t\in Z^m} f(X_t) = Y$. So we can take $U = \bigcup_{t\in \Gamma} X_t$. 
\end{proof}


\bl \label{internal} Let $1 \to A \to B \stackrel{f}\to C \to 1$ be an exact sequence of definable group homomorphisms in $(Z,R)^{eq}$ and assume that $A$ and $C$ are $R$-internal. Then $B$ is $R$-internal.  \el 

\bp By Lemma \ref{Zsat} there is an $R$-internal subset $U$ of $B$ such that $f|U: U \to C$ is surjective. 
Thus $B = U \cdot \ker(f)$. Since $U$ and $\ker (f)$ are $R$-internal, it follows that $B$ is $R$-internal. 
\ep 

We can now obtain the following corollary from the main theorem. 

\bc \label{centr} Under the hypothesis of Theorem \ref{main2} we have:
\begin{enumerate}
\item $G/C_G(\Gamma)$ is $Z$-internal.
\item $\Gamma\cdot C_G(\Gamma)$ is a subgroup of finite index of $G$. 
\item If $G$ has no non-trivial $Z$-internal quotients (this may be regarded as a notion of connectedness), then $\Gamma$ is included in the center of $G$. 
\item $G/Z(\Gamma)$ is a direct product of an $R$-internal and a $Z$-internal subgroup. 
\end{enumerate}
\ec  

\bp
(1) Consider the action of $G$ on $\Gamma$ by conjugation. This gives a morphism from $G$ to $Aut(\Gamma)$ whose kernel is $C_G(\Gamma)$. Thus $G/C_G(\Gamma)$ can be identified with a definable family of automorphisms of $\Gamma$, and since $\Gamma$ is $Z$-internal it easily follows that $G/C_G(\Gamma)$ is $Z$-internal. 

(2) The group $\frac G {\Gamma \cdot C_G(\Gamma)}$ is both $Z$-internal and $R$-internal (being a quotient of both $G/\Gamma$ and $G/C_G(\Gamma)$), so it must be finite. 

(3) If $G$ has no non-trivial $Z$-internal quotients, then $G/C_G(\Gamma)$ must be trivial (by (1)), so $\Gamma$ is in the center of $G$. 

(4) We have $Z(\Gamma) = \Gamma \cap C_G(\Gamma)$. So we have an isomomorphism $G/Z(\Gamma) \cong G/\Gamma \times G/C_G(\Gamma)$ and point (4) is established. 
\ep 

\section{O-minimal case}\label{o-minimal}
 
In this section we show (Theorem \ref{c0}) that if $R$ is an o-minimal structure and $Z$ is arbitrary, then every group $G$ definable in $(Z,R)$ admits a unique ``$t$-topology'' in analogy with the o-minimal case \cite{Pillay1988}. In particular, if $R$ is based on the reals, then $G$ has a natural structure of a real Lie group. Moreover any $Z$-internal subset of $G$ is discrete in the t-topology. If we additionally assume that $Z$ is superstable of finite Lascar rank we can then apply Theorem \ref{main2} to show that any group $G$ definable in $(Z,R)$ is a cover of a group definable in $R$ (Corollary \ref{cover}). Here by ``cover'' we mean a definable morphism which is also a local homeomorphism in the t-topologies (we do not require $G$ to be connected). 

In \S \ref{dim-sets} we defined a dimension on $(Z,R)^{eq}$ based on the projection on the $Z$-coordinates and using the $\SU$-rank on $Z$. In this section we introduce another dimension based on the o-minimal dimension on $R$. Throughout the section, with the exception of Corollary \ref{cover}, $R$ is o-minimal and $Z$ is arbitrary. 

\bd Given a definable set $X \subseteq Z^m \times R^n$, we define $\dim_R(X)$ as the o-minimal dimension of the projection of $X$ to $R^n$. For $X$ definable in $(Z,R)^{eq}$ we define $\dim_R(X)$ as the o-minimal dimension of $X^R$, where $X^R$ is as in Definition~\ref{XZ}. \ed

\bprop \label{Rdim} Work in $(Z,R)^{eq}$, where $R$ is o-minimal. We have:
\begin{enumerate}
\item (Additivity) Given a definable surjective function $f: X \to Y$ with fibers of constant $R$-dimension $k$, we have $\dim_R(X) = \dim_R(Y) + k$.  
\item (Monotonicity) $\dim_R(X\cup Y) = \max(\dim_R(X), \dim_R(Y))$. 
\item (Dimension zero) $\dim_R(X) = 0$ if and only if $X$ is $Z$-internal. 
\item (Definability) Given $d\in \N$ and a definable
family~$(X_t : t\in Y)$ defined over~$A$, the set $Y_d :=
\{t\in Y\;|\;\dim_R(X_t)=d\}$ is definable over~$A$. 
\end{enumerate}
\eprop
\bp This is similar, {\em mutatis mutandis}, to the proof of the corresponding points in Theorem \ref{dim}, but with the roles of $Z$ and $R$ exchanged. Indeed the relevant points of Theorem \ref{dim} (namely everything with the exception of the ``Density'' property) do not use the stability assumption, but only the fact that we have a good notion of dimension on the relevant sort. So we can carry out exactly the same arguments with the o-minimal dimension instead of the $SU$-rank. \ep 


\bd Given a definable set $X$ and a definable subset $Y \subseteq X$, we say that $Y$ is {\em $R$-large} in $X$ if $\dim_R(X\setminus Y) < \dim_R(X)$ (with the convention that if $\dim_R(X) = 0$ this means that $X = Y$). Equivalently, $Y$ is $R$-large in $X$ if $Y$ intersects every $R$-internal subset of $X$ of maximal dimension into a large set. We say that a point $a\in X$ is $R$-generic over a set of parameters $A$ if $a$ belongs to every $R$-large subset of $X$ defined over $A$. 
\ed
 
\bl \label{Rlarge} Let $G$ be a definable set in $(Z,R)$ with $R$ o-minimal. If $X \subseteq G$ is $R$-large in $G$, then $X$ is generic, namely finitely many translates of $X$ cover $G$. \el
\bp 
Choose a small model $M_0$ over which $G$ is defined. It suffices to prove that $G = G(M_0) \cdot X$. 
By orthogonality $G$ has an $R$-internal definable subset $Y$ of maximal $R$-dimension which is defined over $M_0$ (one of the sections over the $Z$-coordinates). Now if $U$ is $R$-large in $G$, then $U \cap Y$ is large in $Y$, and since $Y$ is $R$-internal,  it contains a tuple from $M_0$. We have thus proved that every $R$-large subset of $G$ contains a point from $M_0$.  
To finish the proof, let $X \subseteq G$ be $R$-large and let us show that $G=G(M_0)\cdot X$. To this aim let $g\in G$. Then $g\cdot X^{-1}$ is $R$-large in~$G$ by the invariance of $\dim_R$ under definable bijections (which follows from Proposition \ref{Rdim}), hence it has a point $\gamma$ from $M_0$. 
It follows that $g\in\gamma\cdot X\subseteq G(M_0)\cdot X$ and since $g\in G$ was arbitrary we obtain $G = G(M_0) \cdot X$.   
\ep 
 
\bd Let $R$ be o-minimal. We put on $R$ the topology generated by the open intervals, and on $Z$ the discrete topology. 
Given a subset $X \subseteq Z^m \times R^n$, we define the {\em ambient topology} on $X$ as the topology it inherits from $Z^m \times R^n$, where the cartesian powers have the product topology. 
\ed

When $X$ is a definable group, in addition to the ambient topology, we will introduce also a group topology on $X$, called the  {\em t-topology}, which will be defined by a suitable modification of the construction in \cite{Pillay1988}. See also \cite{Marikova2007,Wencel2012,Fornasiero2012} for similar results. Our proof is self contained. 

\bl \label{cont} Let $X$ and $Y$ be definable sets in the two-sorted structure $(Z,R)$. 
Let $f\colon X\to Y$ be a definable function. Then the set of continuity
points of~$f$, with respect to the ambient topology, is $R$-large in~$X$.
\el
\bp The result is well known in the o-minimal case and since the notion of $R$-largeness refers to the $R$-internal sets we can readily reduce to that case. \ep 

\begin{theorem}\label{c0}
Let $G$ be a definable group in $(Z,R)$, with $R$ o-minimal and $Z$ arbitrary. Then $G$ has a unique group topology, called the t-topology, which coincides with the ambient topology on an $R$-large open subset $V$ of $G$. 
\end{theorem}
\begin{proof} By Lemma \ref{cont} there is an $R$-large subset $Y$ of $G \times G \times G$ such that the function $\alpha\colon (x,y,z)\mapsto xyz$ from $G\times G \times G$ to $G$ is continuous on $Y$ (with respect to the ambient topology). Replacing $Y$ with its interior we can assume that $Y$ is open in $G \times G \times G$ (still with respect to the ambient topology). Note that if $(x,y,z)\in G \times G \times G$ is generic, then it belongs to $Y$. 
Let $V_0$ be the set of all $x\in G$ such that for all $(g_1,g_2)\in
G\times G$ which are $R$-generic over $x$ both $(g_1,x,g_2)$ and $(g_1,g_1^{-1}xg_2^{-1},g_2)$ belong to $Y$.
By the definability of dimension $V_0$ is definable.
Moreover $V_0$ contains all $R$-generic elements, so it is
$R$-large in $G$. Let $V$ be the interior of $V_0$ in $G$. Then $V$ is definable, $R$-large, and open in $G$.

\begin{claim} For all $a,b\in G$, $Z := V \cap aVb$ is open
in~$V$, and the function~$f\colon x\mapsto a^{-1}xb^{-1}$ from~$Z$ to~$V$
is continuous.\end{claim}
\begin{pclaim} 
Let $z\in Z$. We will show that there is an open
neighbourhood of $z$ contained in~$Z$ and $f$ is continuous at~$z$. Pick $(a_1,b_1)\in
G\times G$ generic over~$a,b,z$. Write $a^{-1}= a_2a_1$ and $b^{-1}= b_1b_2$.
We write $f$ as the composition of
$x\stackrel{g}\mapsto a_1xb_1 \stackrel{h}\mapsto a^{-1}xb^{-1}$ where
$h\colon\zeta\mapsto a_2\zeta b_2$.
Consider the following subsets of~$Z$
\begin{align*}
Z_0 &:= \left\{\zeta\in {V} \phantom{i}\st(a_1,\zeta,b_1)\in Y\right\}\\
Z_1 &:= \left\{\zeta\in Z_0\st(a_2,a_1\zeta b_1,b_2)\in Y\right\}
\end{align*}
Clearly $Z_0$ is open in $V$, $g$ is continuous on~$Z_0$, and $z\in Z_0$.
From this it follows that $Z_1$ is open in $V$ and $f$ is continuous on~$Z_1$. Moreover $z\in Z_1$
observing that $(a_2,a_1zb_1,b_2) = (a_2,a_2^{-1}a^{-1}zb^{-1}b_2^{-1},b_2)$ and $(a_2,b_2)$ is
generic over~$a,b,z$.
Now, $f|_{Z_1}^{-1}(V)\subseteq Z$ is an open neighbourhood of~$z$, and $f$
is continuous at~$z$. 
\end{pclaim}
We define the t-topology. A subset $O$ of~$G$ is \textit{t-open} iff
for all $a,b\in G$ the subset~$aOb\cap V$ of~$V$ is open in~$V$. 

By the previous claim $V$ is t-open. More generally we have: 
\begin{claim} 
$O\subseteq V$ is open in~$G$ (with respect to the ambient topology) if and only if it is t-open. In other words the t-topology and the ambient topology coincide on $V$. 
\end{claim}
\begin{pclaim}
Clearly a t-open subset of~$V$ is open in~$V$. Conversely suppose that $O \subseteq V$ is open in $V$. We must prove that $aOb\cap V$ is open in $V$. We have $aOb\cap V = a(O\cap a^{-1}Vb^{-1})b = aO'b$ where $O' := O\cap a^{-1}Vb^{-1}$. Note that $O'$ is an open subset of $V$ because it can be written as the intersection of $V\cap a^{-1}Vb^{-1}$ (which is open by the previous claim) and $O$. To prove that $aO'b$ is open it suffices to observe that $aO'b = f^{-1}(O')$ where $f$ is the continuous function of the previous claim. 
\end{pclaim}
Now we prove that the group operation is t-continuous. The group translations are clearly t-continuous. Since $V$ is $R$-large, there are $a,b,c\in V$ such that $ab=c$. By Lemma \ref{cont} the group operation $\mu$ is continuous at $(a,b)$ with respect to the ambient topology, hence also t-continuous (since the two topologies coincide on $V$). To prove t-continuity at another point $(x,y)\in G\times G$, we go from $(x,y)$ to $(a,b)$ by the t-continuous map $(ax^{-1}(\cdot), (\cdot)y^{-1}b)$, then from $(a,b)$ to $ab$ by $\mu$, and finally from $ab$ to $xy$ via $xa^{-1}(\cdot)b^{-1}y$. 

Finally we show that the group inverse $x\mapsto x^{-1}$ is t-continuous. Consider an $R$-generic point $a\in G$. Then $a^{-1}$ is also $R$-generic, so both $a$ and $a^{-1}$ belong to $V$. By Lemma \ref{cont} the group inverse is continuous at $a$, hence t-continuous. To prove continuity at another point $b\in G$, note that $x\mapsto (x b^{-1}a)^{-1}$ is t-continuous at $b$ and $x^{-1} = b^{-1}a(x b^{-1}a)^{-1}$ is obtained by composing with a group translation.  

We have thus proved the existence of a group topology which coincides with the ambient topology on an $R$-large open subsets. Granted the existence, the uniqueness is clear.
\end{proof}

In analogy with the o-minimal case we have: 

\bprop \label{closed} Let $G$ be a definable group
 in $(Z,R)$ with $R$ o-minimal. Every definable subgroup $H$ of $G$ is closed in the t-topology.
\eprop
\bp First note that the t-closure of a definable subset of $G$ is definable. Indeed the closure of a definable set in the ambient space is definable and we can reduce to this case working in an $R$-large subset $V$ of $G$ where the two topologies coincide. So, replacing $H$ with its closure, it suffices to show that a dense subgroup $H$ of $G$ coincides with $G$. So assume that $H$ is dense. By o-minimality of $R$, a dense subset of $V$ has interior. It follows that $H$ has interior and therefore it is t-open in $G$. Being also a dense subgroup, it coincides with $G$. \ep 

We next show that the t-topology of a quotient coincides with the quotient topology. 

\bt \label{quotient} Let $f: G \to H$ be a definable surjective homomorphism of definable groups in $(Z,R)$, with $R$ o-minimal. Then $f$ is continuous and open with respect to the t-topologies. So the t-topology of $H$ can be identified with the quotient topology of $G/\ker(f)$. 
\et
We need the following easy lemma, true in arbitrary o-minimal structures: 
\bl \label{open} Work in an o-minimal structure. Let $f: X\to Y$ be a definable continuous surjective map. 
Then there is an open subset $U$ of $X$ such that $f|U: U \to Y$ is an open map. \el
\bp Let $G(f)$ be the graph of $f$. Then $X$ is definably homeomorphic to $G(f)$. Replacing $X$ with $G(f)$ we can assume that $f$ is a projection on some coordinates. So we must prove that if $Z$ is a definable subset of $X\times Y$ and $\pi: Z \to X$ is the projection on $X$, then there is an open subset $U$ of $Z$ such that $f|U : U \to X$ is an open map. To this aim take a cell decomposition of $Z$ such that the projections of the cells of $Z$   give a cell decomposition of $X$. Let $C$ be a cell of maximal dimension of $X$. Then $\pi^{-1}(C)$ is an open subset of $Z$ and $\pi$ restricted to this set is an open map. \ep 

\bp[Proof of Theorem \ref{quotient}] Consider the restriction of $f$ to an open $R$-internal subset $U$. 
By o-minimality $f|U$ is continuous at some point. Since the group translations are continuous, $f$ is continuous everywhere. 
We prove that $f$ is open. First note that, by the additivity of dimensions, if $Y$ is a definable subset of $H$ then $\dim_R(Y) + \dim_R(\ker(f)) = \dim_R(f^{-1}(Y))$. Considering the complements it follows that the image of an $R$-large subset of $G$ is $R$-large in $H$.
To prove that $f$ is open it suffices to prove that it is open at some point. Consider an open large subset $U$ of $G$. Then $f(U)$ is large in $H$ and therefore it has interior in $H$ (since $H$ has an $R$-large open subset in the t-topology and two $R$-large sets have a non-empty intersection). We can then take a definable subset $V$ of $f(U)$ which is open in $H$ and consider the restriction $f|f^{-1}(V) : f^{-1}(V) \to V$. By Lemma \ref{open} $f|f^{-1}(V)$ is open at some point in the ambient topologies of its domain and image. But the t-topology coincide with the ambient topology on this sets, so $f$ is open at some point in the t-topology. 
\ep 

We can now introduce a notion of connectedness for definable groups in $(Z,R)$. This motivates the parenthetical remark in Corollary \ref{centr}(3). 

\bd Let $G$ be a definable group in $(Z,R)$. We say that $G$ is {\em t-connected} if $G$ satisfies one of the equivalent conditions of Proposition \ref{conn} below. \ed 

\bprop \label{conn} Let $G$ be a definable group in $(Z,R)$, with $R$ o-minimal. The following are equivalent: 
\begin{enumerate}
\item $G$ has no proper clopen definable subsets in the t-topology. 
\item $G$ has no proper open definable subgroups in the t-topology. 
\item $G$ has no definable subgroups $H$ such that $G/H$ is $Z$-internal. 
\item $G$ has no definable subgroups $H$ with $\dim_R(H) = \dim_R(G)$. 
\end{enumerate} 
\eprop
\bp 
Let us first prove the equivalence of (1) and (2). So let $X$ is a clopen definable subset of $G$. It suffices to show that its stabilizer $\textrm{Stab}(X)= \{g: g X \subseteq X\}$ is an open subgroup of $G$ (note that $\textrm{Stab}(X) \neq G$ if $\emptyset \neq X \neq G$). To this aim it suffices to observe that any t-connected open neighborhood of the identity of $G$ must be contained in $\textrm{Stab}(X)$. 

To finish the proof it suffices to prove the equivalence of the following conditions:  
\begin{enumerate}[(i)]
\item $H$ is an open subgroup in the t-topology of $G$. 
\item $\dim_R(H) = \dim_R(G)$. 
\item $G/H$ is $Z$-internal. 
\end{enumerate}
To show that (i) implies (ii), let $V$ be a large subset of $G$ where the ambient topology coincides with the t-topology. If $H<G$ is open, then $H$ intersects $V$ is an open set, so it has maximal $R$-dimension. The implication from (ii) to (i) is easy. The equivalence of (ii) and (iii) follows from the additivity of the $R$-dimension (Proposition \ref{Rdim}). 
\ep 

\br \label{discrete} Every $Z$-internal subset of $G$ is discrete in the t-topology. \er
\bp 
The t-topology has a basis of $R$-internal sets. Since an $R$-internal set can intersect a $Z$-internal set in at most finitely many points, each $Z$-internal subset is discrete.  
\ep

So far in this section $Z$ was an arbitrary structure. For the next corollary we also need a stability assumption to be able to apply Theorem \ref{main2}. 

\bc \label{cover} If $R$ is o-minimal and $Z$ is superstable of finite Lascar rank, then every group $G$ definable in $(Z,R)$ is a ``cover'' of a group definable in $R$, namely there is a definable morphism $f: G \to H$ such that $H$ definable in $R$ and $f$ is a local homeomorphism (equivalently $\ker(f)$ is $Z$-internal). 
\ec
\bp By Theorem \ref{main2} there is a $Z$-internal subgroup $\Gamma$ such that $G/\Gamma$ is interpretable in $R$ (hence,  by \cite{Ramakrishnan2011}, definably isomorphic to a group $H$ definable in $R$). Moreover the morphism $G \to G/\Gamma$ has a discrete kernel (by Remark \ref{discrete}) and it is continuous and open (by Theorem \ref{quotient}), so it is a local homeomorphism. 
\ep 

In the ``classical'' situation when $R$ is an o-minimal expansion of the reals and $Z = (\Z, +)$ it follows that every (connected) group $G$ definable in $((\Z,+),R)$ is a cover, in the classical sense, of a real Lie group definable in $R$. The proof of this result was the intial motivation for our work. 

Let us finish by mentioning some related work on locally definable groups in o-minimal structures. In 
\cite{Eleftheriou2012} it is proved that if $G$ is a locally definable connected abelian group in an o-minimal structure $R$ and $G^{00}$ exists, then $G$ is a cover of a definable group (and therefore it is interpretable in $((\Z,+), R)$ by \cite{Hrushovski2011}). However in the non-abelian case the corresponding result fails \cite[Example 7.1]{Berarducci2012}, while Theorem \ref{cover} holds also in the non-abelian case. In general the class of the locally definable connected groups (even assuming that $G^{00}$ exists) is much larger than the class of groups interpretable in $((\Z,+),R)$. One way to see this is to observe that every group interpretable in $((\Z,+),R)$ has the non-independence property (NIP), while by \cite{Mamino} there are (connected) groups definable in $((\Z,+),(\R,+,\cdot))$ which interpret the ring of integers (or even the real field with a predicate for the integers).

\bibliographystyle{plain}

\end{document}